\documentclass{amsart}
\usepackage{latexsym, amsfonts, amsthm, amsmath, amssymb, tikz, enumitem, booktabs, float,color,multicol,epsf, url,epsfig,epstopdf, array, setspace, graphicx, csquotes, xcolor,tikz-cd}

\usepackage{setspace} 
\usepackage{pinlabel}
\usepackage{amsfonts}
\usepackage{amsmath}
\usepackage{amssymb}
\usepackage[applemac]{inputenc}
\usepackage{url}
\usepackage{hyperref}
\usepackage{color}
\usepackage{lipsum}
\usepackage{mathtools}
\usepackage[T1]{fontenc}
\usepackage{tikz-cd}
\usepackage{url} 
\usepackage{hyperref} 
\usepackage{pdfpages}
\usepackage{graphicx}
\usepackage{ulem}

\usepackage{tikz-cd}
\usetikzlibrary{shapes.geometric}

\newcommand\blfootnote[1]{%
	\begingroup
	\renewcommand\thefootnote{}\footnote{#1}%
	\addtocounter{footnote}{-1}%
	\endgroup
}

\usepackage{graphicx}

\def\C{\mathbb{C}}
\def\R{\mathbb{R}}

\def\H{\mathbb{H}}
\def\Z{\mathbb{Z}}

\def\F{\mathcal{F}}

\def\N{\mathbb{N}}
\def\F{\mathcal {F}}

\def\T{\mathbb{T}}

\def\Heis{{\sf{Heis}} }
\def\Der{{\sf{Der}}}

\def\heis{{\mathfrak{heis}} }

\def\Aff{\sf{Aff}}

\def\L{\mathcal L}

\def\S{\mathbb S}

\def\L{{\sf{L}}}

\def\N{{\mathcal{N}}}
\def\Aut{{\sf{Aut}}}

\def\PSL{{\sf{PSL}}}
\def\GL{{\sf{GL}}}
\def\O{{\sf{O}}}
\def\SO{{\sf{SO}}}

\def\SL{{\sf{SL}}}

\def\Span{{\sf{Span}} }

\def\Mink{{\sf{Mink}} }

\def\AdS{{\sf{AdS}} }

\def\det{{\sf{det}}}

\def\sl{{\mathfrak{sl}}}

\def\g{{\mathfrak{g}}}

\newtheorem{theorem}{{Theorem}}[section]
\newtheorem{proposition}[theorem]{{Proposition}}%[section]
\newtheorem{isom.ext}[theorem]{{Trivial isometric extension}}%[section]
\newtheorem{definition}[theorem]{{Definition}}%[section]
\newtheorem{lemma}[theorem]{{Lemma}}%[section]
\newtheorem{corollary}[theorem]{{Corollary}}%[section]
\newtheorem{fact}[theorem]{{\sc Fact}}%[section]
\newtheorem{remark}[theorem]{{Remark}}%[section]
\newtheorem{question}[theorem]{{Question}}%[section]
\newtheorem{example}[theorem]{{Example}}%[section]
\definecolor{purple}{rgb}{0.65,0.12,0.94}
\definecolor{forestgreen}{rgb}{0.4,0.64,0.13}
\hypersetup{
	colorlinks=true,
	linkcolor=red,
	filecolor=magenta,      
	urlcolor=black,
	citecolor=blue
}

\begin{document}
%\onehalfspacing
	\title[]{Introduction to Kundt spaces}\footnote{To appear in the  Proceedings of the  \href{https://www.gelomer2024.net}{XI International Meeting on Lorentzian Geometry} (GeLoMer2024).}

	\author [L. Mehidi]{Lilia Mehidi}
	\address{Departamento de Geometria y Topologia
		\hfill\break\indent
		Facultad de Ciencias, Universidad de Granada, Spain}
	\email{lilia.mehidi@ugr.es
		\hfill\break\indent
\url{https://mehidi.pages.math.cnrs.fr/siteweb/}}
	
	\author[A. Zeghib]{Abdelghani Zeghib }
	\address{UMPA, CNRS, 
		\'Ecole Normale Sup\'erieure de Lyon, France}
	\email{abdelghani.zeghib@ens-lyon.fr 
		\hfill\break\indent
\url{http://www.umpa.ens-lyon.fr/~zeghib/}}
	
	\date{ \today}
	\maketitle

	\begin{abstract}
		This paper provides an introduction to Kundt spaces, clarifying several important properties, many of which are typically scattered across the mathematical literature or presented without explicit reference to Kundt terminology. While not exhaustive, our approach aims to offer a pedagogical introduction, using a more geometric language and focusing on key concepts directly related to these spaces, such as lightlike totally geodesic foliations.
	\end{abstract}
	\tableofcontents
	
	\section{Introduction}
	
	Kundt spaces play a significant role in General Relativity and mathematical physics. They constitute the underlying (as well as a unifying) structure for many Lorentzian spaces arising in General Relativity and alternative gravity theories, including plane waves, pp-waves, Siklos spaces.. 
	Our motivation here is to provide an introduction to the mathematics of Kundt spaces through a coordinate-free approach, which is hard to find in the General Relativity literature. We gather and present important properties and key concepts related to those spaces, often dispersed across the mathematical literature and sometimes presented without use of this terminology. This text is not intended to be exhaustive; several aspects are intentionally left out. 
	
	\subsection{Some vocabulary}
	We introduce here some vocabulary:
	\begin{definition}
		Let $M$ be a Lorentzian manifold. A vector subspace $E$ of the tangent space $T_x M$ at $x$ is said to be  lightlike if the restriction of the metric to $E$ is  degenerate. \\
        By analogy, we say that a submanifold $S \subset M$ is lightlike if for any point $x \in S$, $T_x S$ is  lightlike. 
	\end{definition}

	Given a non-singular vector field $V$ on $M$ which is lightlike (i.e. $V(x)$ is lightlike, for all $x \in M$), the distribution $V^\perp$ is lightlike, and the quotient bundle $V^\perp/\R V$ inherits a Riemannian metric. 
	\begin{definition}
		Let $M$ be a Lorentzian manifold, and let $V$ be a non-singular lightlike  vector field on $M$. The vector field $V$ is said to be \textbf{geodesic} if it satisfies $\nabla_V V=0$. In other words, the parameterized orbits of $V$ are geodesics.	
	\end{definition}
	A  non-singular geodesic lightlike vector field defines a family of lightlike geodesics that fill the space, forming what is known as a \textit{null geodesic congruence}. The following terminology can be found  in \cite{penrose}, in the chapter entitled `null congruences', or in \cite{exact-solutions}: 
	\begin{definition}
		Let $M$ be a Lorentzian manifold. Let $V$ be a non-singular lightlike  vector field on $M$. Then $V$ is said to be
		\begin{enumerate}
			\item \textbf{Twist-free} if the distribution $V^\perp$ is integrable, hence tangent to a foliation $\F$. In this case, the leaves of $\F$ are lightlike, in which case we say that the foliation is lightlike.  
			\item \textbf{Shear-free} if the flow of $V$ preserves the conformal Riemannian structure on $V^\perp/V$. 
			\item \textbf{Divergence-free} if the flow of $V$ preserves the volume form of $V^\perp/V$. 
		\end{enumerate}
	\end{definition}
	A conformal map  preserving the volume is necessarily an isometry. 
	Consequently, $V$ is twist-free, shear-free, and divergence-free if and only if $V^\perp$ is integrable, and the (local) flow of $V$ (and hence the local flow of any vector field tangent to $V$) preserves the Riemannian metric induced on $V^\perp /V$. 
	The foliation $\mathcal{V}$, defined by such a vector field $V$, is said to be `transversally Riemannian' when restricted to any leaf of the foliation $\F$. 
	
	\subsection{Kundt spaces}
	A Kundt space is defined in the literature as a Lorentzian manifold with a non-singular lightlike vector field $V$ which is geodesic, twist-free, shear-free, and divergence-free (see, for instance, \cite{coley1, coley2, coley3}). For a detailed account of their significance in General Relativity, we refer to the book \cite{exact-solutions}, a comprehensive reference on exact solutions to Einstein's field equations in General Relativity. This book discusses Kundt spaces, including the properties of null geodesic congruences and their role in defining Kundt geometries.
	\bigskip
	
	The conditions of being twist-free, shear-free, and divergence-free, or equivalently, the `transversally Riemannian' property of the foliation $\mathcal{V}$ when restricted to any leaf of $\mathcal{F}$, turns out to be equivalent to another geometric property: $\mathcal{F}$ being a lightlike, totally geodesic foliation. This equivalence was observed in \cite{zeghibgeo}, where codimension one lightlike totally geodesic foliations are studied. In Section \ref{Section: Generalities: the mathematics behind Kundt spaces}, we review this equivalence as the first key result, enabling Kundt spaces to be characterized in terms of totally geodesic foliations.  This key concept is recalled below:
	
	\begin{definition}
		Let $M$ be a semi-Riemannian manifold.  A submanifold $L \subset M$ is totally geodesic if any geodesic in $M$ which is
		somewhere tangent to $L$ is locally contained in $L$: if $\gamma(t)$ is a geodesic of $M$ defined in a neighborhood of $0$, and if $\gamma'(0) \in T_{\gamma(0)} L$, then $\gamma(t) \in L$ in a neighborhood of $0$. Equivalently, the space of $C^1$ vector fields tangent to $L$ is invariant under covariant derivation (this equivalence holds for any torsion-free connection). 
	\end{definition}
	
	\begin{definition} A totally geodesic foliation $\F$ of $M$ is a foliation whose leaves are totally geodesic submanifolds of $M$. 
	\end{definition}
	
	In Section \ref{Section: Local vs Global}, we derive the local form of the metric for a Kundt space. By specializing to certain simplified metric forms, we identify several well-known subfamilies, including Brinkmann spaces and Siklos spaces.
	
	\subsection{Topology and dynamics} 
	It is well known that a smooth compact manifold admits a Lorentzian metric if and only if  it has zero Euler number, as discussed in \cite[Proposition 37, p. 149]{BO}. However, the existence of a Kundt structure introduces additional topological obstructions, due to the presence of a lightlike totally geodesic foliation. This is addressed in Section \ref{Section: Global topology of Kundt spaces}, starting with the straightforward $2$-dimensional case and extending to the $3$-dimensional case. Another interesting question in this context is: what dynamics lead to a Kundt (or locally Kundt) structure? Here, `dynamics' refers to the action of the isometry group. 
    The result we will review  in the final section is that, in the homogeneous case,  a ``big'' isotropy group leads to a structure that closely resembles a locally Kundt structure. \bigskip

    \paragraph{\textbf{Acknowledgment}} We would like to thank the referee for the valuable comments and suggestions that helped improve the quality of the presentation. 
	
	\section{Generalities: the mathematics behind Kundt spaces}\label{Section: Generalities: the mathematics behind Kundt spaces}
	\subsection{Totally geodesic lightlike foliations}
	Codimension $1$ totally geodesic and lightlike foliations are studied in \cite{zeghibgeo}. As shown in Proposition \ref{Proposition: tot geod foliation iff trans. Riemannian} below,  such a foliation is characterized by the fact that it contains a subfoliation of dimension $1$ whose restriction to each leaf is transversally Riemannian. \\
    
    \textbf{Notation.} Throughout this paper, for  a Lorentzian manifold $(M,\F)$ with a codimension $1$ lightlike  foliation $\F$, we denote by $\mathcal{V}$ the $1$-dimensional subfoliation tangent to $T \F^\perp$. This subfoliation will be referred to as \textit{the normal foliation}.

\begin{proposition}\label{Proposition: tot geod foliation iff trans. Riemannian}
Let $(M,\F)$ be a Lorentzian manifold with a codimension $1$ lightlike foliation. Then $\F$ is totally geodesic if and only if its normal foliation $\mathcal{V}$ is leafwise transversally Riemannian, that is, any vector field tangent to $\mathcal{V}$ preserves the degenerate Riemannian metric on $T\F$.  
\end{proposition}
\begin{proof}
This is precisely the equivalence between items (1) and (2) in Lemma    \ref{Lemma: equivalent items, tot geod iff V Killing along the leaves} below. 
\end{proof}
        
This proposition allows us to give the following equivalent definition of Kundt spaces, first pointed out and adopted in \cite{boucetta}. 
	
\begin{definition}
A \textbf{Kundt space} is a Lorentzian manifold with a non-singular \textbf{geodesic and lightlike} vector field $V$, such that the orthogonal distribution $V^\perp$ is integrable, tangent to a foliation $\F$ which is \textbf{totally geodesic}.   
\end{definition}
We also introduce the following broader class
\begin{definition}
A \textbf{locally Kundt space} is a Lorentzian manifold with a codimension $1$ totally geodesic lightlike foliation $\F$. 
\end{definition}
    
We now state and prove the key lemma underlying Proposition \ref{Proposition: tot geod foliation iff trans. Riemannian} and this reformulation. 
	\begin{lemma}\label{Lemma: equivalent items, tot geod iff V Killing along the leaves}
		Let  $(M,g,V)$ be a Lorentzian manifold with a non-singular lightlike vector field $V$.  The following are equivalent:
		\begin{enumerate}
			\item The distribution $V^\perp$ is integrable, tangent to a totally geodesic foliation $\F$.
			\item The distribution $V^\perp$ is integrable. Moreover, for any leaf $F$ tangent to $V^\perp$, the local flow of any vector field collinear to $V$ preserves the degenerate Riemannian metric induced on $F$.
			\item The line field $\R V$ is parallel along any curve tangent to $V^\perp$.
			\item There exists a differential $1$-form $\alpha$ such that $\nabla_X V = \alpha(X) V$ for any $X \in \Gamma(V^\perp)$.
		\end{enumerate}
	\end{lemma}
	\begin{proof}
		(1) $\Longleftrightarrow$ (2):  Let $Z$ be a vector field collinear to $V$. Item (2) means that for any vector field $X$ tangent to $F$ and invariant by the flow of $Z$, $g(X,X)$ is constant along the orbits of $Z$, that is, $\nabla_Z \; g(X,X)=0$.  Let $X$ be such a vector field.
		We have 
        \begin{align*}
        Z (g(X,X)) = 2 g(\nabla_Z X, X)= 2 g(\nabla_X Z, X) = -2 g(Z, \nabla_X X),    
        \end{align*}
        where we used the facts that $[Z,X]=0$ and $g(Z,X)=0$ in the intermediate steps.  Now,  $F$ is geodesic if and only if $\nabla_X X$ is tangent to $F$ for any $X \in \Gamma(TF)$.  This is equivalent to the condition $g(Z, \nabla_X X)=0$ for all such $X$. From the previous equalities, we see that this is equivalent to $Z(g(X,X))=0$, and thus to the fact that the flow of $Z$ preserves the restriction of $g$  to $TF$. \\
		(1) $\implies$ (3):  Let $F$ be a leaf of $\F$. By assumption, the distribution $TF$  is parallel along any curve contained in $F$. Consequently, when we parallel transport a line tangent to $V$ along a curve in $F$, we obtain a lightlike line tangent to $F$. This line must be  equal to $\R V$, since $\R V$ is the unique lightlike line tangent to $F$.\\
		(3) $\implies$ (4): Let $X$ be a vector field tangent to $V^\perp$. When we parallel transport $V$ along an integral curve of $X$, we obtain a vector field collinear to $V$, i.e. of the form $fV$, where $f$ is a function along the curve.  This implies that $\nabla_X V$ is also collinear to $V$. Consequently, we can write $\nabla_X V= \alpha(X) V$, where $\alpha$ is a $1$-form given by  $\alpha:=g(\nabla V, U)$, with $U$ being a globally defined vector field such that $g(V,U)=1$.  \\
		(4) $\implies$ (1): Let $X$ and $Y$ be two vector fields tangent to the distribution $V^\perp$. We have $g(\nabla_X Y, V)=-g(Y, \nabla_X V)=-\alpha(X) g(Y,V)=0 $. This implies that $V^\perp$ is integrable and tangent to a totally geodesic foliation. 
	\end{proof}
	
In particular, Kundt spaces include Lorentzian manifolds that admit a parallel lightlike line field, or have a lightlike Killing field with an integrable orthogonal distribution.\medskip

\textbf{Convention:} From now on, we assume that the manifold, the foliation and the subfoliation are orientable.  \medskip

From Lemma   \ref{Lemma: equivalent items, tot geod iff V Killing along the leaves}, we also obtain a criterion for a codimension $1$ foliation to be lightlike, totally geodesic for some Lorentzian metric. We recall that a $C^{k}$ foliation is one for which the transition maps of the defining atlas are of class $C^{k}$. 

\begin{corollary}\label{Corollary: criterion}
A given $C^{k+1}$ foliation $\F$ on a manifold $M$ is lightlike geodesic for some $C^k$ Lorentzian metric on $M$, if there is a $1$ dimensional subfoliation $\mathcal{V}$, of class $C^{k+1}$, which is leafwise transversally Riemannian for some $C^k$ degenerate Riemannian metric on $T\F$.    
\end{corollary}
\begin{proof}
Let $h$ be the $C^k$ degenerate Riemannian metric on $T\F$. Define a codimension $2$ distribution $E$ on $M$ such that $E$ is tangent to $\F$ and transversal to $\mathcal{V}$.  Let $V$  be a non-singular vector field tangent to $\mathcal{V}$,  and let $Z$ be a vector field transversal to $\F$. Such objects exist by our orientability assumption, and can be defined using an auxiliary Riemannian metric on $M$. Next, define a Lorentzian metric $g$ on $M$ by setting $g=h$ on $T \mathcal{F}$, and  $g(X,Z)=0$, $g(V,Z)=1$, $g(Z,Z)=0$, for any $X \in \Gamma(E)$. Then, by Lemma   \ref{Lemma: equivalent items, tot geod iff V Killing along the leaves}, $V^\perp$ defines a lightlike totally geodesic foliation.     
\end{proof}
In dimension $3$, this yields
	\begin{corollary}[A simple criterion in dimension $3$]\label{Corollary: criterion in dimension 3}
		Let $M$ be a  $3$-dimensional manifold and $\F$ a  foliation by surfaces. %Assume that $\F$ contains a $1$-dimensional sub-foliation $\mathcal{V}$, tangent to $\F$.
        If $\F$ is lightlike geodesic for some Lorentzian metric on $M$, then any unit vector field tangent to $\F$ preserves the $1$-dimensional sub-foliation $\mathcal{V}$ tangent to $T\F^\perp$. Conversely, assume that $\F$ contains a $1$-dimensional subfoliation $\mathcal{V}$ of class $C^{k+1}$. If there exists a $C^k$ vector field tangent to $\F$, transversal to $\mathcal{V}$, and preserving $\mathcal{V}$, then $\F$ is a lightlike geodesic foliation for some Lorentzian metric on $M$ of class $C^k$.     
	\end{corollary}
	
\begin{proof}
Let $F$ be a surface with a one dimensional foliation $\mathcal{V}$. Let $h$ be a degenerate Riemannian metric on $F$ with radical $T \mathcal{V}$, and let $X$ be a unit vector field tangent to $F$.  We claim that $\mathcal{V}$ preserves $h$  (i.e. the local flow of any vector field tangent to $\mathcal{V}$ preserves $h$) if and only if $X$ preserves $\mathcal{V}$. To see this,  denote by $\phi_X^t$  the flow of $X$, and by $d_h$ the semi-distance associated to $h$.  The metric $h$ is $\mathcal{V}$-invariant if and only if $d_h$ is $\mathcal{V}$-invariant, which in turn is equivalent to the following property $(\mathfrak{P})$: 
`for any $\mathcal{V}$-leaf $l_0$, and any two pairs of points $(p,q)$ and $(p',q')$ such that $p$, $p' \in l_0$, the distances $d_h(p,q)$ and $d_h(p',q')$ are equal if and only if $q$ and $q'$ are on the same $\mathcal{V}$-leaf'. 
Let $p \in F$ and $q:=\phi_X^s(p)$. Since $X$ is a unit vector field, the distance $d_h(p, q)=s$ is equal to the time of the flow between the two points. Consequently,  
for any two points $p, p'$ on the same $\mathcal{V}$-leaf $l_0$,  $d_h(p,\phi_X^{s_0}(p))=d_h(p',\phi_X^{s_1}(p'))$ if and only if $s_0=s_1$. Assume now that $h$ is $\mathcal{V}$-invariant. Then, for the distances above to be equal, we must have  $\phi_X^{s_0}(p)$ and $\phi_X^{s_0}(p')$  on the same $\mathcal{V}$-leaf, and this holds for any initial $\mathcal{V}$-leaf $l_0$ and 
$p$, $p' \in l_0$. 
This means that the flow of $X$ maps any $\mathcal{V}$-leaf to a  $\mathcal{V}$-leaf, i.e. that $X$ preserves $\mathcal{V}$. Conversely, if $X$ preserves $\mathcal{V}$, then any other unit vector field also preserves $\mathcal{V}$, since it is necessarily of the form $X+f V$ for some function $f$ and some vector field $V$ tangent to $\mathcal{V}$. Thus, the property $(\mathfrak{P})$ above is satisfied, which means  that the semi-distance $d_h$  is $\mathcal{V}$-invariant. This proves our claim. 
From this claim, it follows that $\mathcal{V}$ is leafwise transversally Riemannian if and only if there exists a vector field tangent to $\F$, transversal to $\mathcal{V}$, and preserving $\mathcal{V}$. 
With this in hand, the corollary is simply a reformulation in dimension $3$ of Proposition \ref{Proposition: tot geod foliation iff trans. Riemannian} and Corollary \ref{Corollary: criterion}.
\end{proof}

	\subsection{Examples}\label{Section. Examples}
	Explicit examples of (locally) Kundt spaces are given here. More examples can be found in Section \ref{Section: Local vs Global}. 
	\begin{example}[Minkowski space]\label{EX: Minkowski}
		The Minkowski space is a Kundt space: any  foliation by lightlike hyperplanes  is a lightlike totally geodesic foliation. %Such a foliation has a tangent lightlike Killing field.     
	\end{example}
	
	\begin{example}[Anti-de Sitter space]\label{EX: AdS}
    		Denote by $\R^{2,n+1}$, $n \geq 1$, the real vector space $\R^{3+n}$ equipped with the quadratic form: $$q:=2 du_1 dv_1 +2 du_2dv_2 + dx_1^2 + \ldots + dx_{n-1}^2$$ of signature $(2,n+1)$. Then we define
$$\AdS^{2+n} := \{x \in \R^{2,n+1} \vert \; q(x) = -1\}.$$
            It is immediate to check that $\AdS^{2+n}$ is a smooth connected submanifold of $\R^{2,n+1}$ of dimension $n+2$. The tangent space $T_x \AdS^{2+n}$ regarded as a subspace of $\R^{3+n}$ coincides with the orthogonal space $x^\perp = \{y \in \R^{n+3} \vert \; q(x, y) = 0\}$. The restriction of the quadratic form $q$ to $T \AdS^{2+n}$ has Lorentzian signature, making $\AdS^{2+n}$ a Lorentzian manifold. This defines the so-called quadric model of anti-de Sitter space of dimension $n+2$. 
          
            The lightlike vector field $V:=u_2 \partial_{v_1} - u_1 \partial_{v_2}$ acts on $\R^{2,n+1}$ by the one-parameter group of diffeomorphisms $f_t\left(u_1,v_1,u_2,v_2,x) =(u_1, v_1 + t u_2,\, u_2,\, v_2 - t u_1, x \right)$, where $x:=(x_1,\ldots,x_{n-1})$. It is easy to see that these diffeomorphisms preserve $q$, so $V$ is a lightlike Killing field for $q$. Moreover, $V$ is tangent to $\mathsf{AdS}^{2+n}$. So, $V$ is also a (lightlike) Killing vector field  of $\mathsf{AdS}^{2+n}$. The distribution $V^\perp$ of $\R^{3+n}$ is integrable, since it is generated by the vector fields $\partial_{v_1}$, $\partial_{v_2}$, $\partial_{x_1}, \ldots, \partial_{x_{n-1}}$, and $u_1 \partial_{u_1} + u_2 \partial_{u_2}$. Thus, the distribution $E:=V^\perp \cap T \mathsf{AdS}^{2+n}$ on the anti-de Sitter space is integrable, tangent to a foliation $\F$. 
            This foliation is lightlike, totally geodesic, and has codimension one in $\mathsf{AdS}^{2+n}$.  
            
        We now express the metric of $\AdS^{2+n}$  in a system of adapted coordinates. This will allow us to define the so-called Siklos spaces in the next section, and see them as  locally Kundt spaces that generalize the anti-de Sitter space.   Consider the half-space $H^{2+n}_-:=\mathsf{AdS}^{2+n} \cap \{u_1 <0\}$, and define the diffeomorphism $\varphi^-: \,\R^{2+n}_{y_n >0}\, \to H^{2+n}_-$, that sends $\left( u,v,y,y_n \right) \in \R^{2+n}_{y_n >0}$ to $\left(u_1,v_1,u_2,v_2,x\right) \in H^{2+n}_-$, with 
            \begin{align*}
                 \left(u_1,v_1,u_2,v_2,x\right)=y_n^{-1} \left(- 1,  (y_n^2+  \parallel y \parallel^2 + uv), u , v , y \right),
            \end{align*}
        where  $y:=\left(y_1,\ldots, y_{n-1}\right)$, and $\parallel y \parallel ^2:= \sum_{i=1}^{n-1} y_i^2$.  
         Under this map, the pull back of the $\AdS^{2+n}$ metric takes the form
         \begin{equation}\label{Eq: metric anti-de sitter}
          g^\AdS= (y_n)^{-2}(2du dv + dy_1^2 +\ldots + dy_{n}^2).
         \end{equation}
        Similarly, we can define $\varphi^+: \R^{2+n}_{y_n <0} \to H^{2+n}_+ := \mathsf{AdS}^{2+n} \cap \{u_1 >0\}$. 
        These coordinates $(\ref{Eq: metric anti-de sitter})$ are known as the Poincar\'e coordinates on the (half) anti-de Sitter space. The lightlike hyperplanes $u=$constant form a lightlike totally geodesic foliation,  denoted by $\mathfrak{F}$. Indeed, as it appears from (\ref{Eq: metric anti-de sitter}), its normal foliation, defined by the lightlike vector field $\partial_v$, is transversally Riemannian on each leaf. In particular, the transversal metric induced on each leaf is the hyperbolic metric.    
        Moreover, since $\partial_v$ is mapped by the diffeomorphisms $\varphi^{\pm}$ on $V$, the image of $\mathfrak{F}$ by $\varphi^-$ (resp. $\varphi^+$) coincides with the restriction of the lightlike, totally geodesic foliation $\F$ to the open subset $H^{2+n}_-$ (resp. $H^{2+n}_+$) of $\mathsf{AdS}^{2+n}$. 

        A final observation  in this example is that the foliation on $H^{2+n}_-$ extends to a foliation on the manifold with boundary $\overline{H^{2+n}_-}=H^{2+n}_- \cup \partial H^{2+n}_-$, where the boundary $\partial H^{2+n}_-$ is a leaf of the extended foliation. This phenomenon is not specific to this case but holds more generally, at least locally. Indeed, an important property, shown in  \cite{zeghib-lipschitz}, is that codimension $1$ geodesic foliations of manifolds with a $C^1$ affine connection are locally Lipschitz. As observed there, given such a manifold $Y$, and a foliation $\F$ defined on a subset $X \subset Y$, this property allows to extend $\F$ as a foliation of the closure $\overline{X} \cap U_p$, for some (convex) neighborhood $U_p$ of any point $p \in \partial X$. In this extended foliation, the boundary $\partial X \cap U_p$ becomes a boundary leaf.  
	\end{example}

    \begin{example}[Suspensions]\label{EX: suspensions} Let $(N,\mathcal{S})$ be a Riemannian manifold with a codimension $1$ foliation $\mathcal{S}$. Let $f$ be a diffeomorphism of $N$ that preserves both $\mathcal{S}$ and the induced Riemannian metric on $\mathcal{S}$. The maps $(x,t) \mapsto (f^n(x), t+n)$, $n \in \Z$, define a proper and free action of $\Z$ on the product $N \times \R$. Let $M$ denote the quotient manifold of $N \times \R$ by this action, i.e. $M:=N \times \R/(x,t) \sim (f(x),t+1)$.  This is called the suspension manifold of $f$. The vector field $\partial_t$ on $N \times \R$ is invariant under this action, so it determines a nowhere vanishing vector field $V$ on $M$. The flow of $V$ is called the suspension flow of $f$, and defines a foliation which we denote by  $\mathcal{V}$. 
    Now, define $\F$ as the saturation of $\mathcal{S}$ by $\mathcal{V}$. By Corollary \ref{Corollary: criterion}, $\F$ is lightlike geodesic for some Lorentzian metric on $M$, with $\mathcal{V}$ as a normal foliation. Note in particular that this metric is not a product of metrics on $N$ and $\R$.   
	\end{example}
	
\begin{example}[A torus bundle over the circle]\label{EX: Solvable FG}
Let $q$ be a Lorentzian quadratic form on $\R^2$, and consider the flat Lorentzian torus $(\T^2 = \R^2/\Z^2,q)$.  Let  $h \in \O_\Z(q) := \O(q) \cap \GL(2,\Z)$ be a hyperbolic matrix,
i.e. $h$ is real diagonalizable with eigenvalues different from $\pm 1$. The matrix $h$ naturally defines an action on $\T^2$. %Let $h^t \subset \O(q) $ be a one-parameter group such that $h^1 = h$.
Consider the suspension manifold of $h$, defined as $M:=\T^2 \times \R/ (x, t) \sim (h(x), t+1)$. Endow $M$ with the Lorentzian (flat) metric $g$ induced from the product metric $q + dt^2$ on $\T^2 \times \R$. Let $e_1$ be an eigenvector of $h$. It  determines a foliation $\mathcal{V}$ on $\T^2$, and $h$ acts as a diffeomorphism of $\T^2$ that preserves this foliation. Define $\F$ as the saturation of $\mathcal{V}$ by the suspension flow of $h$. 
The foliation $\F$ is lightlike and totally geodesic with respect to the metric $g$, with $\mathcal{V}$ as a normal foliation.    
\end{example}

\begin{example}[Compact quotients of $\AdS^3$]\label{Ex: non-solvable FG} Consider the one-parameter subgroups of $\SL(2,\R)$ given by
		\begin{align*}
			d^t= \begin{pmatrix}
				e^t & 0 \\
				0 & e^{-t}
			\end{pmatrix}, \;
			h^t= \begin{pmatrix}
				1 & 0 \\
				t & 1
			\end{pmatrix}.    
		\end{align*}
We will give, for $n=1$, another description of the Kundt structure on $\AdS^{2+n}$ described in  Example \ref{EX: AdS}, and obtain such a structure also on (compact) quotients of $\AdS^{3}$. 
There is a special model of anti-de Sitter space in dimension $3$ which naturally endows it with a Lie group structure. To construct this model, consider the vector space $\mathcal{M}(2, \R)$ of $2 \times 2$ real matrices. Then, $q := - \mathsf{det}$ is a quadratic form of signature $(2, 2)$, so there is an isomorphic identification of $(\mathcal{M}(2,\R),-\mathsf{det})$ with $\R^{2,2}:=(\R^{2+2},2du_1dv_1+2du_2dv_2)$ (unique up to composition by an element in $\O(2, 2)$). Under this isomorphism, the anti-de Sitter space $\AdS^3$ is identified with the Lie group $\SL(2, \R)$. 

We now determine the Lorentzian metric induced by this identification. The tangent space at the identity element $\text{I}_2$ of $\SL(2,\R)$ is   identified with the Lie algebra  of $\SL(2,\R)$, given by $\sl(2,\R):=\left\{ A=\begin{pmatrix}
			a & b \\
			c & -a
\end{pmatrix}; a, b, c \in \R \right\}$. The  quadratic form $q$ induces a quadratic form $\bar{q}$ on $\sl(2,\R)$ given by $\bar{q}(A)=-\det(A)=a^2 - bc$. This form has Lorentzian signature. 
The group $\SL(2, \R) \times \SL(2, \R)$ acts linearly on $\mathcal{M}(2, \R)$ by left and right multiplication:
$$(A,B) \cdot X := AXB^{-1}.$$  This action preserves both 
the quadratic form $q = - \mathsf{det}$ and the subspace $\SL(2, \R)$. Thus,  it induces an isometric action on $\SL(2, \R)$ endowed with the  Lorentzian metric induced by $q$, and this action corresponds to left and right multiplications. Therefore, with this model, one has a natural identification of $\AdS^3$  with $(\SL(2, \R), g_{\AdS})$, where  $g_{\AdS}$ is the Lorentzian metric obtained by left-translating the quadratic form $\bar{q}$ on $\sl(2,\R)$. This metric is also right-invariant. In particular, the right action of 
$h^t$ on $\SL(2,\R)$ generates a lightlike Killing field on it, which, up to applying an element of $\O(2,2)$ on $\R^{2,2}$, coincides with the one described in Example \ref{EX: AdS}. 

Now, for any uniform lattice $\Gamma$ in $\SL(2, \R)$, we have an induced Lorentzian metric  on the quotient manifold $X:=\Gamma \textbackslash \SL(2,\R)$, on which $\SL(2, \R)$ acts isometrically on the right. 	The right action of $h^t$ generates a lightlike Killing field on $X$, and its orthogonal distribution defines a lightlike, totally geodesic foliation $\F$. 
This foliation is given by the right action of $D \ltimes H$ on $X=\Gamma \backslash \SL(2,\R)$, where $D$ and $H$ are the one-parameter subgroups of $\SL(2,\R)$ defined by $D:=\{d^t, t \in \R\}$, $H:=\{h^t, t \in \R\}$.
\end{example}

\subsection{Actions of Lie groups}
Certain cohomogeneity $1$ actions of Lie groups induce codimension $1$ foliations that can be made lightlike and totally geodesic for some Lorentzian metric.  Some of the examples discussed in the previous paragraph are of this type:
\begin{itemize}
    \item In Example \ref{EX: suspensions}, let $N$ be a $2$-dimensional torus, and let $V$ denote the vector field that generates the suspension flow. Consider $X$, a unit vector field tangent to $\mathcal{S}$. Since $f$ preserves both $\mathcal{S}$ and the Riemannian metric induced on it,  $X$ is preserved by $f$. Define a vector field $\hat{X}$ on $N \times \R$  by setting $\hat{X}(x,t)=X(x)$. This vector field  commutes with the horizontal  translation vector field $\partial_t$.  Since $X$ is preserved by $f$, $\hat{X}$ descends to  the quotient $M$,  defining a non-vanishing vector field  on $M$ that is tangent to $\F$, transversal to $\mathcal{V}$, and commutes with  $V$. Hence, the foliation  $\F$ is induced by a (locally free) action of $\R^2$. 
    \item In contrast to the previous example, the gluing diffeomorphism   $h$ in Example \ref{EX: Solvable FG} preserves a $1$-dimensional foliation of $\T^2$ but it does not preserve a parametrization of its leaves.  In this example, the foliation $\F$ is defined by a locally free action of  the affine group.  To see this, let $h^t$ be a one-parameter group of hyperbolic matrices in $\mathsf{O}(q)$ such that $h^1 = h$. Let $e_1$ and $e_2$ be  two distinct eigenvectors of $h$.  We have $h(e_1)= \lambda e_1$, \, $h(e_2)=\lambda^{-1} e_2$, for some $\lambda \in \R^*$, $\lambda \neq \pm 1$. 
    Then, for all $t \in \R$, the matrices $h^t$ have  the same eigenvectors, and we have $h^t(e_1) = \lambda^t e_1$,\, $h^t(e_2) = \lambda^{-t} e_2$.  As seen in Example \ref{EX: Solvable FG},  the choice of an eigenvector determines a $2$-dimensional foliation $\F$ on $M$, which is tangent to the direction field generated by the chosen eigenvector on $\T^2$ and to the suspension flow of $h$. Fix an eigenvector, say $e_1$, and define two vector fields on $\T^2 \times \R$ by
    \begin{align*}
    	X(x,t)= \lambda^t e_1,\;\; Y(x,t)= \frac{1}{\ln \lambda} \partial_t.
    \end{align*}
   Both descend to the  quotient $M$, and they are tangent to $\F$. Moreover, they satisfy the Lie bracket relation $[Y,X]=X$.  Therefore, these vector fields generate the Lie algebra of the affine group. Since they are complete, they induce a locally free action of the affine group, defining the foliation $\F$. 
    \item In Example \ref{Ex: non-solvable FG}, the right action of $D \ltimes H$ on $X=\Gamma \backslash \PSL(2,\R)$ is locally free, and defines the foliation $\F$. This group is isomorphic to the affine group. 
\end{itemize}

These examples are particular cases of the following general situation. Let $G$ be a Lie group acting locally freely on a manifold $M$ with codimension $1$ orbits. Assume that $G$ has a normal, connected, $1$-dimensional subgroup $H$. Denote by $\F$ the orbit foliation of $G$ and by $\mathcal{V}$ the subfoliation corresponding to $H$. Then

\begin{proposition}\label{Prop: actions of Lie groups}
There exists a Lorentzian metric on $M$ such that $\F$ is lightlike and totally geodesic, with $\mathcal{V}$ as normal foliation.    
\end{proposition}
\begin{proof}
Let $X _0, \ldots , X _d$ be a basis of the Lie algebra of $G$, with $X_0$ corresponding to $H$. They determine fundamental vector fields $X_0,\ldots,X_d$ on $M$, which span $T\F$. On $T\mathcal{F}$, consider the degenerate Riemannian metric defined by: $\langle X_0,X_i \rangle= 0$ for all $i$, and $\langle  X_i,X_j \rangle= \delta_{ij}$ for $i, j \neq 0$. Since $H$ is normal, any bracket $[X_0,X_i]$ is a multiple of $X_0$. Consequently, the flow of $X_0$, corresponding to the action of $H$, maps $X_i$ to a vector field of the form $X_i + f X_0$, for some function $f$. This shows that $\mathcal{V}$ preserves the degenerate Riemannian metric.  The proposition then follows directly from Corollary \ref{Corollary: criterion}. 
\end{proof}

\begin{example}
In addition to Examples \ref{EX: suspensions} (with $N=\T^2$), \ref{EX: Solvable FG} and \ref{Ex: non-solvable FG}, Proposition \ref{Prop: actions of Lie groups} implies that any foliation of a $3$-dimensional manifold defined by a locally free action of $\R^2$ or $\Aff(\R)$ can be made lightlike and totally geodesic with respect to some Lorentzian metric.     
\end{example}

\begin{example}[Oscillator group]\label{EX: oscillator group}
Let $\Heis_3$ denote the $3$-dimensional Heisenberg group, whose Lie algebra $\heis_3 = \C \oplus \R$ is generated by $X, Y, Z$,  with the only non-trivial Lie bracket $[X,Y]=Z$. Denote by $C$ the center of $\Heis_3$, whose Lie algebra is generated by $Z$. 
Consider a semi-direct product $G_\rho := \S^1 \ltimes_\rho \Heis_3$, where $\S^1$ acts on $\Heis_3$ via a morphism $\rho: \S^1 \to \Aut(\Heis_3)$. 
Define a Lorentzian scalar product on the Lie algebra $\g_\rho$  as follows: 
\begin{itemize}
    \item $\langle \;\;,\;\; \rangle:$ the usual Euclidean scalar product on $\C = \R^2$, 
    \item $\langle T,T \rangle = \langle Z,Z\rangle = 0,\langle T,Z\rangle = 1$,
    \item $\R T \oplus \R Z  \perp \R^2$.
\end{itemize}
Consider the left-invariant Lorentzian metric on $G_\rho$ induced by this scalar product. 
The right action of $\Heis_3$ on $G_\rho$ defines a left-invariant codimension $1$ foliation $\F$ of $G_\rho$.  In particular, the $1$-dimensional  foliation $\mathcal{V}$,  generated by the right action of $C$ (the center of $\Heis_3$) is  also left-invariant and therefore lightlike. 
On the other hand, since $C$ is a normal subgroup of $G_\rho$, the  foliation $\mathcal{V}$ coincides with the one defined by the left action of $C$. Indeed, for any $x \in G_\rho$, we have $ x C x^{-1}= C$, hence $xC=Cx$. So the left action of $C$ generates a vector field tangent to $\mathcal{V}$. This vector field  preserves the degenerate Riemannian metric induced on any leaf of $\F$. 
Therefore, the foliation $\mathcal{V}$ is transversally Riemannian on each leaf of $\F$. 
Consequently, $\F$ is lightlike, totally geodesic, with normal foliation given by $\mathcal{V}$. 
In the special case where the $\rho$-action of $\S^1$ on $\Heis_3$ is trivial on the center of $\Heis_3$ and acts by rotation on the $\C$-factor, the resulting Lorentzian group $G_\rho$ has a bi-invariant symmetric Lorentzian metric, and is known as the \textbf{oscillator group}. It is a special example of a \textbf{Cahen-Wallach space} (see Par. \ref{Par. CW space}). These spaces are important in the study of symmetric Lorentzian manifolds.
\end{example}
    
\section{Anosov flows and Kundt spaces of low regularity}
	\subsection{From Anosov flows to locally Kundt structures}The last two examples in the previous paragraph belong to a more general family of foliations, arising as the weak stable foliation of an Anosov flow.
	Recall that a non-singular flow $\phi^t$ of class $C^\infty$
	on a compact $3$-manifold $M$ is called Anosov  if there exists a decomposition of the tangent bundle $TM$ into a direct sum of three rank-$1$ subbundles, denoted by $E^s$, $E^u$ and $E$, such that the following properties are satisfied:
	\begin{enumerate}
		\item $E$ is the line bundle of the flow,
		\item $E^s$ and $E^u$ are invariant under $\phi^t$,
		\item The stable distribution $E^s$	is uniformly contracted by $\phi^t$, and the unstable distribution $E^u$	is uniformly expanded by $\phi^t$. 
	\end{enumerate}
	The plane fields $E \oplus E^s$ and $E \oplus E^u$ define, respectively, what is called the weak stable and weak unstable foliations of the Anosov flow. So  the weak stable  foliation contains a $1$-dimensional subfoliation $\N^s$, which is tangent  to the stable distribution $E^s$. Similarly, the weak unstable foliation contains a subfoliation $\N^u$ tangent to $E^u$, and the Anosov flow, which is tangent to the weak stable and the weak unstable foliations, preserves both $\N^s$ and $\N^u$. This structure suggests that one could apply Corollary \ref{Corollary: criterion in dimension 3} to conclude that the weak stable and weak unstable foliations of an Anosov flow are lightlike geodesic for some Lorentzian metric on $M$.  However, this argument faces a critical issue of regularity: for general $(C^\infty)$ Anosov flows, the distributions $E^s$ and $E^u$ are only $C^0$ (see \cite{ghys_flotsAnosov}). In fact, it is shown in \cite[Theorem 7]{zeghibgeo} that if the weak stable foliation of an Anosov flow on a compact $3$-manifold is lightlike geodesic for some $C^\infty$ metric, then, up to finite cover, it is $C^\infty$ diffeomorphic to the weak stable foliation of an algebraic Anosov flow, i.e. up to finite cover, the flow is either the geodesic flow of a compact hyperbolic surface $\Sigma$, acting on its unitary tangent bundle $T^1 \Sigma$, or it is the suspension of a hyperbolic linear diffeomorphism of the $2$-torus. The last two examples mentioned in Par. \ref{Section. Examples} are precisely of this algebraic type:
	\begin{itemize}
		\item In Example \ref{EX: Solvable FG}, the flow of translations $\phi^t(x,s)=(x, s+t)$ acts on $M=\T^2 \times \R/(x,s) \sim (h(x),s+1)$ as an Anosov flow. It is the suspension flow of a hyperbolic linear diffeomorphism $h$ of $\T^2$. This diffeomorphism preserves two $1$-dimensional foliations on $\T^2$. And the weak stable and unstable foliations on $M$ are defined by saturating these  foliations   with the suspension flow.  
		\item In Example \ref{Ex: non-solvable FG}, the foliation $\F$ is given by the right action of $D \ltimes H$ on $X=\Gamma \backslash \SL(2,\R)$, where $D$ and $H$ are the one-parameter subgroups of $\SL(2,\R)$ defined by $D:=\{d^t, t \in \R\}$, $H:=\{h^t, t \in \R\}$.  An easy computation gives $$d^t h^s d^{-t} = h^{\exp(-2t)s}  \text{\;\;and\;\;} d^t (h^s)^\top d^{-t} = (h^{\exp (2t) s})^\top.$$ 
		This shows that $d^t$ acts as an Anosov flow on $X$. The weak stable foliation is given by $\F$, and the weak unstable foliation is defined by the right action of $D \ltimes H^\top$ on $X$, where $H^\top :=\{(h^t)^\top, \,t \in \R\}$. Algebraically, the unit tangent bundle of the $2$-hyperbolic space is identified with the group $\SL(2,\R)/\{\pm \text{I}_2\}=\PSL(2,\R)$. 
       One can assume that $\Gamma$ is torsion-free, up to passing to a finite index subgroup. This ensures that the (left) action of $\Gamma$ on the homogeneous space $\SL(2,\R)/\SO(2)$ is proper and free. Since this homogeneous space is identified with the $2$-hyperbolic space $\H^2$, the quotient $\Gamma \backslash \SL(2,\R)/\SO(2)$ defines a hyperbolic surface $\Gamma \backslash \H^2$. Algebraically, the unit tangent bundle of the $2$-hyperbolic space is identified with the group $\SL(2,\R)/\{\pm \text{I}_2\}=\PSL(2,\R)$, so the unit tangent bundle of the hyperbolic surface $\Gamma \backslash \H^2$ is identified with $\Gamma \textbackslash \SL(2,\R)$. The flow $d^t$ introduced above corresponds precisely to the geodesic flow on the unit tangent bundle of the hyperbolic surface $\Gamma \textbackslash \H^2$ (for a proof, see, for instance, \cite[Chap. 9, Par. 9.2]{ergodic}).
	\end{itemize}
	
	As mentioned above, in both examples, the stable and unstable distributions are $C^\infty$. In particular, the existence of a Kundt structure on them also follows from Corollary \ref{Corollary: criterion in dimension 3}. Observe that the Lorentzian metric yielding a Kundt structure is not unique.  
	
	\subsection{Kundt spaces of low regularity}
	The definition of a Kundt space involves a Lorentzian metric and a codimension $1$ lightlike foliation $\F$ that is totally geodesic. To consider totally geodesic foliations, a $C^1$ metric is required. However, the property that there exists a $1$-dimensional subfoliation $\mathcal{V}$, whose restriction to each leaf is transversally Riemannian, only requires a $C^0$ metric. This property can be formulated as follows: for any points $p_1, p_2, q_1, q_2$ on a leaf of $\F$, and for any two curves $\gamma_1$ and $\gamma_2$ joining $p_1$ to $p_2$ and $q_1$ to $q_2$ respectively, if $p_1$ and $q_1$ lie on the same $\mathcal{V}$-leaf, then $\gamma_1$ and $\gamma_2$ have the same length if and only if  $p_2$ and $q_2$ lie on the same $\mathcal{V}$-leaf. 
	This allows us to define a low-regularity locally Kundt structure as a manifold with a low-regularity Lorentzian metric (which may be $C^0$), admitting a codimension $1$ lightlike foliation $\F$ with the aforementioned property.
	
	\begin{remark}[Hyperbolic $3$-manifolds]\label{Remark: weak stable foliation Anosov}	
    In \cite{foulon}, the authors construct many examples of hyperbolic $3$-manifolds admitting Anosov flows with low regularity for their  stable and unstable foliations (see \cite[Theorem 6.2]{foulon}), giving rise to low regularity Kundt structures.
	\end{remark}

	\section{Local vs Global}\label{Section: Local vs Global}
	\subsection{Adapted coordinates}
	Given a Lorentzian manifold $(M,\F)$ with a codimension $1$ lightlike  foliation $\F$, we denote by $\mathcal{V}$ the $1$-dimensional subfoliation tangent to $T \F^\perp$. 
	
	\begin{proposition}\label{Proposition: Kundt adapted coordinates}
		A Lorentzian $(n+2)$-dimensional manifold $(M,\F)$ with a codimension $1$, lightlike totally geodesic  foliation $\F$ admits local coordinates adapted to the foliation, in which the metric has the following form
		\begin{align}\label{Eq: Kundt metric}
			g=2 du dv + H(u,v,x) du^2 + \sum_{i=1}^{n} W_i(u,v,x) du dx^i + \sum_{i,j} h_{ij}(u,x) dx^i dx^j.
		\end{align}
	\end{proposition}
	
	\begin{proof}
		Consider a local $(n+1)$-submanifold $\Sigma$ in $M$ that is transversal to $\mathcal{V}$ (and therefore also to $\F$). The foliation $\F$ induces an $n$-dimensional foliation on $\Sigma$. Let $(u,x_1,\ldots,x_n)$ be coordinates on $\Sigma$ such that the leaves of $\F_{\vert \Sigma}$ are given by the levels of $u$. 
		Define (locally) the vector field $V$ along $\Sigma$ such that it is tangent to $\mathcal{V}$ and satisfies $g(\partial_u,V)=1$. For each point $x \in \Sigma$, consider the lightlike geodesic with initial velocity $V_x$. These geodesics, parametrized by $v$, define a local flow on $M$. Denote again by $V$ the infinitesimal generator of this flow.  Using the flow of $V$,  extend the coordinate vector fields on $\Sigma$ to the saturation of the latter by the flow of $V$.   We obtain $n+2$ commuting vector fields 
		$$U:=\partial_u, V:=\partial_v, X_i:=\partial_{x_i},$$
		satisfying $g(V,V)=0$ and $\nabla_V V=0$. 
		Now, observe that $g(U,V)$ is constant along the integral curves of $V$. Indeed, $V(g(U,V))=g(\nabla_V U, V) + g(U, \nabla_V V)= g(\nabla_U V, V) + 0 =0$. Since $g(U,V)$ is constant along $\Sigma$, it follows that $g(U,V)$ is constant everywhere. Finally, the fact that $\F$ is totally geodesic is equivalent to $\mathcal{V}$ being transversally Riemannian on every leaf of $\F$, which in turn is equivalent to the fact that the functions $h_{ij}$ do not depend on $v$. 
	\end{proof}

	\subsection{Hierarchy}
	Kundt spaces with special form of coordinates  give rise to well-known classes of Lorentzian spaces, having some special geometry on the leaves:

	\subsubsection{\textbf{Brinkmann spaces.}}
	Brinkmann spaces are Kundt spaces for which the foliation $\F$ is tangent to a distribution $V^\perp$, where $V$ is a (global) lightlike parallel  vector field. \medskip
	
	For the metric (\ref{Eq: Kundt metric}), one can see that $V:=\partial_v$ is parallel if and only if  the functions $H$ and $W_i$ do not depend on $v$. Hence the following equivalent definition 

    \begin{fact}
     A Brinkmann space $M$ is a Lorentzian manifold admitting a global vector field $V$, such that any point of $M$ admits a coordinate chart $(u, v, x_1, \ldots , x_n)$ where the metric takes the form  
	\begin{align}\label{Eq: Brinkmann coordinates}
		g=2 du dv + H(u,x) du^2 + \sum_{i=1}^{n} W_i(u,x) du dx^i + \sum_{i,j} h_{ij}(u,x) dx^i dx^j,
	\end{align}
	with $V=\partial_v$.   
    \end{fact}
    
     These coordinates are  known as Brinkmann coordinates. 

\begin{example}\label{EX: homogeneous Brinkmann}
Let $G:=\R \ltimes \Heis_3$ denote a semi-direct product, where $\R$ acts on $\Heis_3$ by a one-parameter group of automorphisms $\rho(t)=e^{t A}$, for some derivation $A \in \Der(\heis_3)$. 
Let $(X, Y, Z)$ be a basis of $\heis_3$, such that $[X,Y]=Z$. Extend it to a basis $(T,X,Y,Z)$ of the Lie algebra $\g$. 
Define a Lorentzian scalar product on $\g$ as follows
\begin{itemize}
    \item $\langle \;\;,\;\; \rangle:$ the usual Euclidean scalar product on $\R^2=\Span_\R(X,Y)$, 
    \item $\langle T,T \rangle = \langle V,V\rangle = 0,\langle T,V\rangle = 1$,
    \item $\R T \oplus \R V  \perp \R^2$.
\end{itemize}
This scalar product induces a left-invariant Lorentzian metric $g$ on $G$. We consider the Lorentzian space $(G,g)$. 
Denote by $V$ the right-invariant vector field on $G$ generated by $Z$. 
Since $V$ is right-invariant, its flow corresponds to left multiplications by a one-parameter subgroup of $G$. This action is isometric, since the metric is left-invariant, and thus, $V$ is a Killing field. The Koszul formula for three Killing fields $U, V, W$ is given by 
$$2g(\nabla_U V, W)=g([U,V],W) + g([V,W],U) - g([W,U],V).$$
Applying this to $V$, one  shows that $\nabla V=0$, meaning that $V$ is a parallel vector field. In particular, it also lightlike. Hence, the space $(G,g)$ is a homogeneous Brinkmann space.     
\end{example}
    
	Compact Brinkmann spaces exhibit interesting geometric properties: their geodesic completeness and the dynamics of the lightlike parallel flow  are studied in \cite{mehidi2022completeness}.
	
	\subsubsection{\textbf{Weakly Brinkmann}}
	These are  Lorentzian manifolds admitting a lightlike parallel  line bundle. In other words, they are locally Kundt spaces in which the line field $T \F^\perp$ is parallel on $M$. 
	These spaces are  sometimes referred to as Walker manifolds, which more broadly refer to  pseudo-Riemannian manifolds with a  parallel lightlike distribution \cite{Walker}. 

    \begin{fact}
      The metric of a weakly Brinkmann manifold has the form (\ref{Eq: Kundt metric}), where the functions $W_i$ do not depend on $v$. The parallel lightlike line field  is then given by $\R \partial_v$.   
    \end{fact}
	
	\begin{proof}
		  We use the same notations as in the proof of Proposition \ref{Proposition: Kundt adapted coordinates} for the coordinate vector fields. Assume  that the line field $T \F^\perp$ is parallel on $M$, then $T \mathcal{F}$ is also parallel. Consequently, $\nabla_U X_i \in \Gamma(T\F)$, so\, $g(\nabla_U X_i, V)=0$. On the other hand, applying the Koszul formula yields $2 g(\nabla_U X_i, V)= -V (g(U,X_i))=- \partial_v W_i$, which yields $\partial_v W_i=0$, hence the first implication. Conversely, suppose that the $W_i$'s do not depend on $v$ (equivalently, that $\nabla_U X_i \in \Gamma(T\F)$ for all $i$). Since $g(X_i,V)=0$, we have $g(\nabla_U X_i, V)=- g(X_i, \nabla_U V)$, which implies, by our assumption, that $g(X_i, \nabla_U V)=0$ for all $i$. On the other hand, $2g(V,\nabla_U V)=U (g(V,V))=0$. So $\nabla_U V$ is orthogonal to $\F$, hence collinear to $V$. Combined with the fact that $\R V$ is parallel along the leaves of $\F$ (see Lemma \ref{Lemma: equivalent items, tot geod iff V Killing along the leaves}), this proves that $\R V$ is a parallel lightlike line field. 	
	\end{proof}
    
We wish to highlight a subtle phenomenon that occurs here, and that does not occur in Riemannian signature. Let $(M,l)$ be a weakly Brinkmann space, with $l$ a lightlike parallel line field on $M$.  Locally, it is always possible to define parallel lightlike vector fields spanning $l$. If $l$ were timelike or spacelike, then, by passing to a time-orientable cover of $M$, one would obtain a global parallel vector field spanning $l$ by taking a constant-length section that is compatible with the orientation. %This is possible in Riemannian signature. 
However,  when $l$ is lightlike,  there is no natural way to select a global parallel section. In this case, time orientation does not help to patch together the existing local parallel sections. Let us give a concrete example.  

\begin{example}
Consider again the Lorentzian space defined in Example \ref{EX: homogeneous Brinkmann}. Keeping the same notations, we have $[T,\omega]=A(\omega)$ for all $\omega \in \heis_3$. 
Assume that  $A(Z)=Z$, so that $[T,Z]=Z$. For $\omega \in \g$, let $\overline{\omega}$ denote   the right-invariant vector field generated by $\omega$.  We claim that the line field $\R \overline{Z}$ is left-invariant, but the vector field $V:=\overline{Z}$ is not left-invariant. Let us prove this. 
Since $[\omega,Z]$ is collinear to $Z$ for all $\omega \in \g$, the (right-invariant) vector field $[\overline{\omega},V]$ is collinear to $V$ for all $\omega \in \g$.  This implies that the line field $\R V$ is left-invariant, since the flow of a right-invariant vector field corresponds to left multiplication by a one-parameter subgroup of $G$.  
On the other hand, the bracket $[T,Z]=Z$ yields $[\overline{T}, V]=V$. Thus, the vector field $V$ is not invariant by the left action of the $\R$-factor in $G$. This proves the claim. 
Now, consider the quotient space $X:=\Gamma \backslash G$, where $\Gamma$ is a discrete torsion-free subgroup of $G$. Assume that $\Gamma$ has a non-trivial projection to the $\R$-factor. Then it contains elements that act on $V$ by scaling, sending $V$ to $\lambda V$, with $\lambda \neq 1$. Consequently, only the lightlike parallel line field $\R V$ descends to the quotient. Such a quotient is therefore a  weakly Brinkmann space. 
\end{example}

\subsubsection{\textbf{A subclass of Brinkmann spaces: pp-waves.}} 
In a Brinkmann space, the codimension one foliation $\F$ is totally geodesic. Thus, for any vector fields $X$, $Y$ tangent to $\F$, $\nabla_X Y$ is also tangent to $\F$, inducing a connection on $T\F$. 
A pp-wave is a Brinkmann space for which the leaves of $\F$ are flat with respect to this induced connection.
	
\begin{fact}
A pp-wave is a Lorentzian manifold with a global vector field $V$, such that each point admits local coordinates of the form (\ref{Eq: Brinkmann coordinates}), where $W_i=0$ and $h=\sum_{i=1}^n (dx^i)^2$ is the Euclidean metric of $\R^n$. In these coordinates, $V$ is  represented locally by  $\partial_v$. Thus,  the  metric of a pp-wave is locally given by
	\begin{align}\label{Eq: pp-waves coordinates}
		g^\Mink_H = 2 du dv + H(u,x) du^2 + \sum_{i=1}^n (dx^i)^2.
	\end{align}    
\end{fact}

	\begin{proof}
	Consider a geodesic $\gamma(t)$ that is transversal to $\F$ and satisfies $g(\gamma'(t), V)=1$. Let  $(X_1,\ldots,X_n)$ be a frame field  
    along $\gamma$, where each $X_i$ is tangent to $\F$, transversal to $\mathcal{V}$, and satisfies $g(\gamma'(t_0),X_i)=0$ at some point $t_0$.  Since the distribution $T\F$ is parallel on $M$, the parallel transport of $X_i(\gamma(t_0))$ along $\gamma$ defines a vector field along $\gamma$ which is tangent to $\F$ and transversal to $\mathcal{V}$. Therefore,  we may assume that the frame field $(X_1,\ldots,X_n)$ is parallel along $\gamma$. In particular, $g(\gamma'(t),X_i)=0$ for all $t$. 
    We  assume further that the frame is orthonormal with respect to the (degenerate) Riemannian metric on each leaf.
    
    Now, extend the frame field $(X_1,\ldots,X_n)_{\gamma(t)}$ to a  parallel frame field along the leaf of $\F$ containing $\gamma(t)$. Since $\F$ is totally geodesic, this extension defines a frame field in a neighborhood of $\gamma$, which we denote again by $(X_1,\ldots,X_n)$, tangent to $\F$ and parallel along each leaf. 
    Finally, extend the vector field $\gamma'(t)$, defined along $\gamma$, by means of the flows of $V$ and $X_i$ for $i=1,\ldots,n$. We obtain a coordinate frame field $(U,V,X_1,\ldots,X_n)$ in a neighborhood of $\gamma$ that satisfies the aforementioned properties. 

   1) We claim that $g(U,V)=1$. To see this, first note that the flow of $V$ is isometric and preserves both $V$ and $U$. Consequently, we have $V (g(U,V))=0$. Next, using the Koszul formula, we write
   \begin{equation}\label{K}
     2  g(\nabla_{X_i} V, U)= X_i (g(V,U)) + V(g(X_i, U)).
   \end{equation}
Since $X_i$ is also preserved by $V$, we have $V (g(X_i, U))=0$. Moreover, $V$ is parallel, so $\nabla_{X_i} V=0$. Then (\ref{K}) yields $X_i (g(V,U))=0$. Thus, $g(V,U)$ is constant along the leaves of $\F$. Since it is also constant along the geodesic $\gamma$, it must be globally constant, and hence equal to $1$. 
   
   2) Next, we will show that $g(U,X_i)=0$ for any $i=1,\ldots,n$. To begin, we will show that 
		\begin{align}\label{Eq-proof}
			\nabla_U X_i \in \R V.
		\end{align}
		Using successively the facts that $[U,X_i]=0$, $g(U,V)=1$, and that $V$ is parallel, we write $g(\nabla_U X_i, V) = g(\nabla_{X_i} U, V) = -g(U, \nabla_{X_i} V) =0$. It follows that $\nabla_U X_i$ is orthogonal to $V$, and  
        hence can be written as  $\nabla_U X_i = \sum_k f_j X_j + \phi V$, where $f_j$ and $\phi$ are smooth functions, and $f_j$ is zero along $\gamma$ by construction.  We will prove that $f_j$ is constant on the leaves of $\F$. We proceed in two steps, using the flatness of the leaves:
        
      \textbf{a)} By construction,  $\nabla_{X_k} X_j=0$ for all $j,k=1,\ldots,n$. Using this, compute $X_k (g(\nabla_U X_i, X_j)) = g(\nabla_{X_k} \nabla_U X_i, X_j) = g(R(X_k, U) X_i, X_j) $. On the other hand, $g(R(X_k, U) X_i, X_j)  =g(R(X_i, X_j)X_k, U)=0$. It follows that $X_k (g(\nabla_U X_i, X_j))=0$, hence  $f_j=g(\nabla_U X_i, X_j)$ is constant along the integral curves of $X_k$ for any $k$.
      
       \textbf{b)} Similarly, $V(g(\nabla_U X_i, X_j)) = g(\nabla_V \nabla_U X_i, X_j)$, since  $\nabla_V X_j = \nabla_{X_j} V =0$. But $g(\nabla_V \nabla_U X_i, X_j) =g(R(V,U) X_i, X_j) =g(R(X_i,X_j)V, U) =0$. It follows that $V(g(\nabla_U X_i, X_j))=0$, proving that $f_j$ is  constant along the integral curves of $V$.

		Therefore, $f_j$ is constant on each leaf of $\F$, and since $f_j$ vanishes on the curve $\gamma$, which is transversal to $\F$, it follows that $f_j=0$ everywhere, which confirms (\ref{Eq-proof}).  
        
		Now, using (\ref{Eq-proof}), we write $X_j(g(U,X_i))= g(\nabla_U X_j, X_i) =0$.  
		Moreover, it follows from (\ref{K}) that $g(U,X_i)$ is constant on a leaf of $V$.  So finally, $g(U,X_i)$ is constant on a leaf of $\F$. Since it is zero on $\gamma$, which is transversal to $\F$, it must be zero everywhere. 
	\end{proof}
	
	\begin{remark}
 A pp-wave generalizes the Minkowski space. When $H=0$, $g^\Mink_0$ is the Minkowski metric.  
	%	2) Kundt spaces with flat leaves are more general spaces than pp-waves.  
	\end{remark}

	\paragraph{\textbf{Foliations with a tangential structure}} It is well known that a flat affine connection on a smooth manifold  of dimension $n+1$ is equivalent to a $(\Aff(\R^{n+1}), \R^{n+1})$-structure on it. So when $(M,\mathcal{F})$ is a pp-wave, the leaves of $\mathcal{F}$ inherit a tangential affine structure. Moreover, the leaves of $\mathcal{F}$ 
    admit a tangent parallel lightlike vector field $V$,
    and an induced parallel degenerate Riemannian metric with  radical $\R V$.
    This endows the leaves with a special affine geometry, referred to in \cite{article1} as the `affine unimodular lightlike geometry', and denoted by $(\L_{\mathsf{u}}(n), \R^{n+1})$. This is a subgeometry of the affine geometry,  where $\mathsf{L}_{\mathsf{u}}(n)= \O(n) \ltimes \Heis_{2n+1}$ is the subgroup of $\Aff(\R^{n+1})$ preserving  the degenerate Riemannian scalar product $q:=dx_2^2+dx_3^2+\ldots+dx_{n+1}^2$, and the lightlike vector field $\partial_{x_1}$.  This group $\L_{\mathsf{u}}(n)$ is thus called the `affine unimodular lightlike group'.

	\subsubsection{\textbf{Plane waves.}} A plane wave is a specific class of pp-waves, in which the function $H$ in (\ref{Eq: pp-waves coordinates}) is quadratic in the variables $(x^i)$, with coefficients that depend on $u$. Explicitly,  $H(u,x)=x^\top S(u) x$, where $S(u)$ is a symmetric matrix depending on $u$. Thus, the metric of a plane wave can be written locally as
	\begin{align}\label{Eq: plane waves coordinates}
		g= 2 du dv + x^\top S(u) x\, du^2 + \sum_{i=1}^n (dx^i)^2.
	\end{align}
	
	\paragraph{\textbf{The algebra of Killing fields}} Plane waves among Brinkmann spaces can also be characterized through their local isometry algebra of Killing fields. Let 
   $\mathfrak{g}_0$   denote the subalgebra of the isometry algebra that fixes a leaf of the $\mathcal{F}$-foliation and preserves the vector field $V$. From the definition of a pp-wave, it follows that there exists a (faithful) representation $$\pi: \mathfrak{g}_0 \to \mathfrak{o}(n) \ltimes \heis_{2n+1},$$
	where $\mathfrak{o}(n) \ltimes \heis_{2n+1}$ is the Lie algebra of the affine  unimodular lightlike group.  It is a known result (see, for instance, \cite{Blau}) that the Lie algebra of Killing fields of a plane wave contains the Heisenberg algebra $\heis_{2n+1}$ . We also have the converse:
	
	\begin{fact}\cite{article1}
		A plane wave is a  pp-wave for which $\pi(\mathfrak{g}_0)$ contains $\heis_{2n+1}$.     
	\end{fact}

	\subsubsection{\textbf{Cahen-Wallach spaces.}}\label{Par. CW space} A Cahen-Wallach space is a plane wave with adapted local coordinates of the form (\ref{Eq: plane waves coordinates}), where the matrix $S(u)$ is constant (does not depend on $u$) and non-degenerate. These spaces are characterized as indecomposable (reducible) symmetric plane waves.
	
	\subsubsection{\textbf{Siklos spaces: a hyperbolization of  pp-waves}}
	Siklos spaces are   another special class of Kundt spaces, defined similarly to pp-waves, but with a different geometry on the leaves: the induced Riemannian metric on $V^\perp / \R V$ is the hyperbolic metric instead of the Euclidean one.  More precisely, a Siklos space has a local coordinate system given by 
	$$ g^{\AdS}_H =  \frac{2 du dv +  H(u, x) du^2 + \mathsf{euc}_n}{(x^n)^2},$$
	where $\mathsf{euc}_n$ is the Euclidean metric on the variable $x = (x^1, \ldots, x^n)$. For $H= 0$, the metric becomes $g^{\AdS}_0 = \frac{g^\Mink_0}{(x^n)^2}$, which is the $\AdS$-metric,  a hyperbolization of  the Minkowski metric.  A Siklos space is, on the other hand, a ``hyperbolization'' of  pp-waves, as $$g^{\AdS}_H = \frac{g^\Mink_H}{(x^n)^2},$$ 
    and $g^\Mink_H$ is the metric of a pp-wave.

	\begin{remark}
		A metric of the form $g_1 =  e^\sigma g$, where $g$ is Kundt and $\sigma$ is a   function that does not depend on $v$, is also Kundt.	So Siklos spaces are obtained by a special conformal change of pp-waves, within the broader class of Kundt spaces. 
	\end{remark}

	The following diagram sketches the hierarchy mentioned above (the arrows indicate inclusions):
	
	\begin{center}
		\[
		\begin{tikzcd}[sep=large,
			every arrow/.append style = {-stealth, shorten > = 5pt, shorten <=2pt},
			]
			\text{Totally geodesic lightlike foliations}  \\                 &  \text{\textbf{Kundt spaces}}   \ar[lu]\\
			\text{Kundt spaces with flat leaves}\ar[ru] & \text{\textbf{Brinkmann spaces}} \ar[u] & \text{\textbf{.}} \ar[lu]\\
			& \text{pp-waves} \ar[lu] \ar[u] & \text{Siklos spaces} \ar[u] \\
			& \text{Plane waves}   \ar[u] & \text{.}  \ar[u]\\
			& \text{Cahen-Wallach spaces} \ar[u] & \text{.}  \ar[u]\\
			& \text{Minkowski space} \ar[u]& \text{Anti-de Sitter space} \ar[u] \\
		\end{tikzcd}
		\]
	\end{center}
	
	\section{Global topology of Kundt spaces}\label{Section: Global topology of Kundt spaces}
	
	We assume here that the manifold $M$ is compact.  Thurston \cite{thurston_existence} has shown using sophisticated topological tools  
	that a manifold admits a codimension $1$ foliation if and only if its Euler characteristic is zero. Apart from this topological obstruction, codimension $1$ foliations
	are extremely flexible. However, requiring the additional condition that the foliation is lightlike and totally geodesic for some Lorentzian metric is considerably more restrictive regarding the topology of the manifolds on which such foliations can be defined. 
	
	\begin{question}
		Given a compact manifold $M$ with zero Euler characteristic, does it admit a locally Kundt structure for some Lorentzian metric on $M$ ?
	\end{question}

	\begin{question}
		Given a foliation $\F$ on a manifold $M$, is there a smooth Lorentzian metric $g$ on $M$ such that $\F$ is lightlike and totally geodesic for $g$ ?    
	\end{question}
	
	On the other hand, it is natural to ask whether these foliations coincide with those that are geodesic with respect to some Riemannian metric.

	\subsection{Dimension 2} Up to a double cover, a Lorentzian surface is diffeomorphic to a torus. Any $1$-dimensional foliation $\F$ on the $2$-torus can be made lightlike geodesic for some Lorentzian metric  $g$. This can be achieved by defining a supplementary foliation   $\F'$ and taking $g$  such that its lightcones are tangent to $\F$ and $\F'$.  
In contrast, not every foliation on the torus $\T^2$ can be geodesic for a Riemannian metric. In fact, a $1$-dimensional foliation on $\T^2$ is geodesic for a Riemannian metric if and only if it is a suspension foliation. Let us explain this in more details. 
    
In \cite[Chap. IV]{godbillon},  a smooth classification of the foliations of the $2$-torus is given.  A smooth foliation on the $2$-torus either contains a Reeb component or is differentiably conjugate to the suspension foliation of a diffeomorphism of the circle. We recall some definitions and facts from \cite[Chap. IV]{godbillon}. 

\begin{figure}[h!]
	\labellist 
	\small\hair 2pt 
	\endlabellist 
	\centering 
	\includegraphics[scale=0.17]{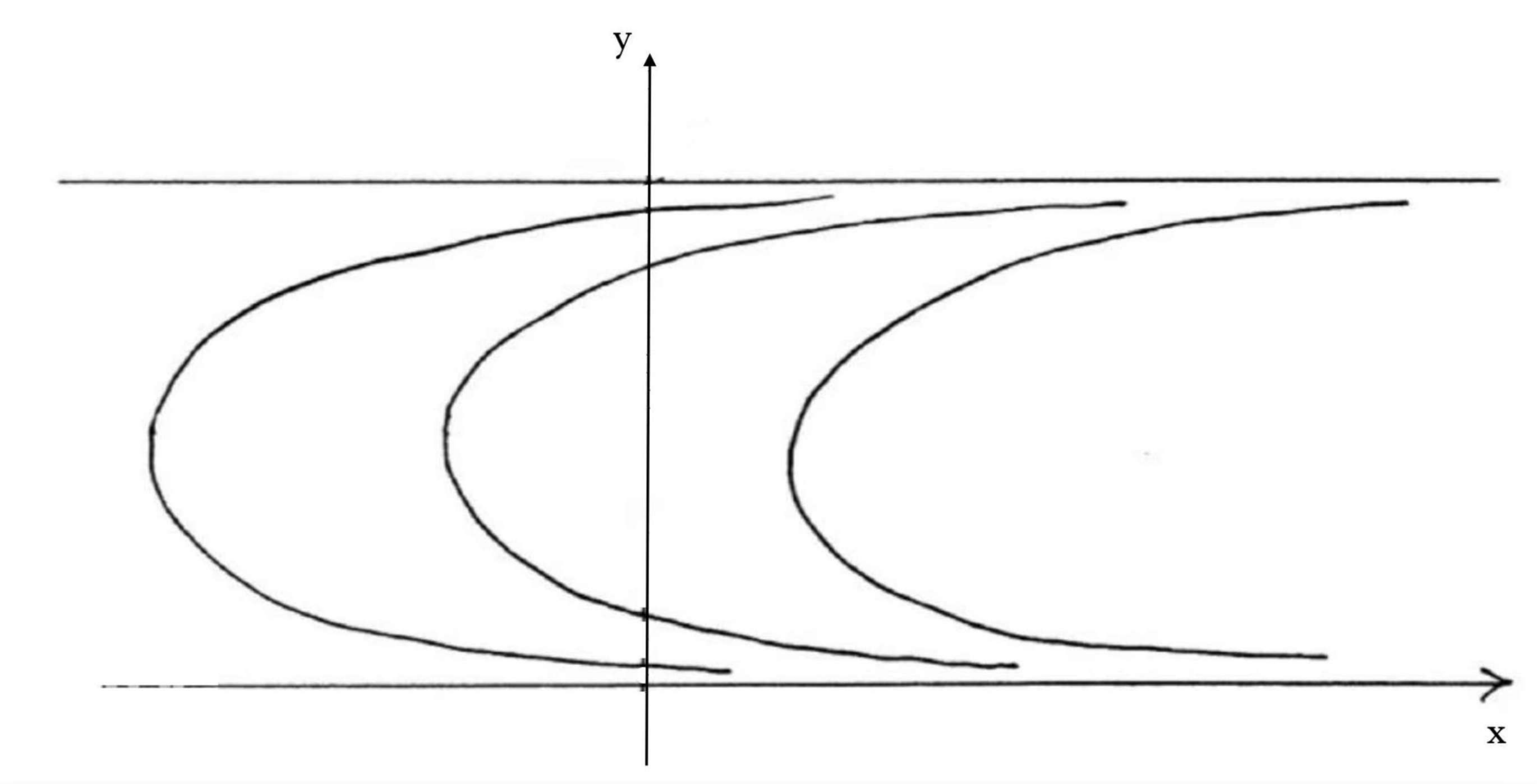} \caption{A Reeb component on $\R^2$. The foliation induced on $\T^2$ by the action of $\Z^2$ on $\R^2$: $(x,y) \mapsto (x+n,y+m)$ is a Reeb component on $\T^2$}
\end{figure}

\paragraph{\textbf{a) Suspension foliations}} Let $f$ be a diffeomorphism of the manifold $N=\S^1$. Let $M:=N \times \R/ (x,t) \sim (f(x),t+1)$ be the suspension manifold of $f$. 
The manifold $M$ is diffeomorphic to the torus $\T^2$ if $f \in \mathsf{Diff}^+(\S^1)$ (i.e. if $f$ preserves orientation), and to the Klein bottle otherwise. 
The foliation on $M$ determined by $\R$ is called the suspension foliation of $f$. 

Any foliation of $\T^2$ that has no Reeb components is smoothly conjugate to a suspension foliation.

The suspension foliation is geodesible for a Riemannian metric on $\T^2$. Such a metric can be constructed as follows. Let $g_0$ be an arbitrary Riemannian metric on $N$. Define $g_1 := f^* g_0$,   the pull back of $g_0$ by $f$, and consider the family of Riemannian metrics on $N$, depending on $t \in \R$, given by $g_t := \frac{1}{2}(1+\cos( \pi t)) g_0 + \frac{1}{2}(1- \cos( \pi t)) g_1$.  Then, the metric $ g_t + dt^2$ on $N \times \R$ induces a metric on $M$. The lines $\{p\} \times \R$, for $p \in N$, are geodesics of $N \times \R$ (while this can be checked directly, one possible geometric argument is that the foliation orthogonal to these lines,  given by the fibers $N \times \{t\}$, $t \in \R$, is, by definition, transversally Riemannian). Therefore, 
the leaves of the suspension foliation are geodesics of $M$.\medskip

\paragraph{\textbf{b) Foliations with Reeb components}} In contrast, a foliation with a Reeb component cannot be totally geodesic for a Riemannian metric. Indeed, it is well known that a foliation $\F$ is totally geodesic with respect to a Riemannian metric if and only if its orthogonal foliation $\mathcal{G}$ is transversally Riemannian. This means the following property: for any two curves $c_1$ and $c_2$ of $\mathcal{G}$, and for any two points $p_1 \in c_1$ and $p_2 \in c_2$ on the same leaf of $\F$, the points $p_1$ and $p_2$ must be equidistant with respect to the distance measured along the leaf of $\mathcal{F}$. However, the presence of a Reeb component in $\mathcal{F}$ implies that $\mathcal{G}$ also has a Reeb component. Hence, any two such curves $c_1$ and $c_2$ have a  limit cycle, which contradicts the previous property.

	\subsection{Dimension 3} 
	\subsubsection{\textbf{Case of a Brinkmann space.}} In this discussion, we consider the case where  $\F$ is the foliation of a $3$-dimensional  compact Brinkmann space; that is, $\F$ is tangent to $V^\perp$, with $V$ being a parallel lightlike vector field for some Lorentzian metric. This setting allows us to derive some topological consequences on the foliation. We will make use of the following fact
    \begin{fact}\label{Fact: 3D Brinkmann is pp-wave}
     In dimension $3$, any Brinkmann space is a pp-wave.   
    \end{fact}
    \begin{proof}
     Recall that since $\F$ is totally geodesic, the restriction of the Levi-Civita connection to $T\F$ induces a connection on $T\F$. 
     In dimension $3$,  a leaf $F$ of $\F$ is a surface with a   tangent lightlike parallel vector field $V$. Such a surface is necessarily flat. To see this, let $X$ be a (local) vector field tangent to $F$ and transversal to $V$. Since $V$ is parallel, we have  $R(X,Y)V=0$, for any $Y \in \Span(X,V)$. Moreover, for any $Y, Z$ tangent to $\F$, $\langle R(X,V)Y, Z \rangle = \langle R(Y,Z)X,V \rangle$, and the latter vanishes since $TF^\perp \subset \R V$. This yields $ R(X,V)Y=0$. This means that the Riemann curvature tensor is identically zero, hence $F$ is flat with respect to the induced connection. 
    \end{proof}
    By Fact \ref{Fact: 3D Brinkmann is pp-wave}, $\F$ inherits a tangential affine unimodular lightlike structure coming from the induced flat affine connection on $T \F$. This structure forces topological obstructions, which are detailed below. We assume that both the manifold and the foliation are orientable.\medskip
	
	A leaf $F$  of the foliation can be homeomorphic to either a plane, a cylinder, or a torus. Indeed,  since compact pp-waves are geodesically complete \cite{leistner2016completeness}, and $\F$ is totally geodesic, $F$ is complete with respect to the induced affine unimodular lightlike structure $(\mathsf{L}_{\mathsf{u}}(1),\R^2)=(\O(1) \ltimes \Heis_3, \R^2)$ (this is completeness in the sense of $(G,X)$-structures). Thus, the universal cover of $F$ develops bijectively onto $\R^2$. In particular, the holonomy representation $\rho_F: \pi_1(F) \to \Heis_3$ (we assumed everything orientable) is injective, and its image acts properly discontinuously and freely on $\R^2$,  which implies that $\pi_1(F)$ is abelian. Therefore,  $F$ must be either a plane, a cylinder or a torus.  \medskip
	
	Furthermore, the foliation has no vanishing cycles. Recall that a vanishing cycle of a foliation $\F$ on $M$ is a mapping $\sigma: \S^1 \times I \to M$  such that:
	\begin{itemize}
		\item for any $s \in I$, the loop $\gamma_s:= \sigma_{\vert \S^1 \times \{s\}}$  lies in a certain leaf of $\F$
		\item for $s \neq 0$, the loop $\gamma_s$ is homotopic to zero in the leaf 
		\item $\gamma_0$ is not homotopic to zero.
	\end{itemize}
	A typical example of a foliation with a vanishing cycle is given by the Reeb foliation on $\S^3$:  a loop  $\gamma$ on the compact fiber $\T^2$ of the Reeb foliation, which represents one of the generators in $\pi_1(\T^2)$, is a vanishing cycle. 
	The holonomy along a vanishing cycle  is trivial. So the presence of a vanishing cycle  in $\F$ contradicts the injectivity of the holonomy representation mentioned above. 
	
	\subsubsection{\textbf{Case of a locally Kundt space}}
	More generally, we have the following result from \cite{zeghibgeo}
	\begin{theorem}$($\cite[Theorem 11]{zeghibgeo}$).$\label{Theorem: 3D universal cover is R3}
		Let $\F$ be a $C^0$ lightlike geodesic foliation in a compact Lorentzian $3$-manifold. Then:
		\begin{itemize}
			\item A leaf of F is homeomorphic to a plane, a cylinder or a torus. 
			\item $\F$ has no vanishing cycles.
			\item The universal cover is homeomorphic to $\R^3$, foliated by planes.
		\end{itemize}    
	\end{theorem}
	
	As a consequence,  $\S^3$ does not admit a locally Kundt structure.

	\begin{remark}
		In dimension $\geq 4$, the universal cover of a compact Kundt space is not necessarily contractible.  For instance, consider the product manifold 
		$M:=\S^1 \times \S^3$ in dimension $4$. Fix a base point $p \in \S^1$, and consider a codimension $1$ foliation $\F$ on $\{p\} \times \S^3$. Choose a vector field $V$ tangent to $\S^1$, which induces an $\S^1$-action on $M$. Using the action of $V$, we extend the foliation $\F$ to a codimension $2$ foliation of $M$. We keep the same notation for the extended foliation. Next, let $h$ be any Riemannian metric on  $ T \F_{\vert \{p\}\times \S^3}$, which we pull back using the $\S^1$-action of $V$ to define a Riemannian metric on $T \F$. Now, take a vector field $U$ on $M$  that commutes with $V$ and is transversal to $\F$. We define a Lorentzian metric $g$ on $M$ as follows: $g( V,V ) = g( U,U ) =0$, $g( U,V ) =1$, and $g=h$ on $T \F$.  Thus, the foliation $\S^1 \times \F$ is lightlike and tangent to $V$. Moreover, by definition, the foliation determined by $V$ is transversally Riemannian on each leaf of $\S^1 \times \F$. Consequently, $\S^1 \times \F$ is totally geodesic for the metric $g$. Here,  $\S^3$ can be replaced by   $\S^{2l+1}$ for any $l \geq 1$, providing a similar example in any (even) dimension. 
	\end{remark}

	\begin{remark}[Fundamental groups of compact  $3$-dimensional Kundt spaces]
		The fundamental group of a compact $3$-dimensional Brinkmann space is (virtually) solvable. However, Example \ref{Ex: non-solvable FG} shows that the fundamental group of a compact Kundt space can be non-solvable. In this particular example, the manifold is a Seifert fiber manifold  with a hyperbolic base (this is shown in \cite{kulkarni}), i.e. it is finitely covered by a circle bundle over a closed orientable surface of genus $\geq 2$. Consequently, its fundamental group is (up to finite index) an extension of a compact surface group of genus $\ge 2$ by an infinite normal cyclic subgroup generated by a regular fiber. On the other hand, by Remark \ref{Remark: weak stable foliation Anosov}, Kundt spaces of lower regularity allow a broader set of possible fundamental groups. 
	\end{remark}

	\section{Dynamics leads to a Kundt structure}\label{Section: Dynamics leads to a Kundt structure}
	
	This section deals with homogeneous Lorentzian manifolds. We will see that a ``big'' isotropy group leads to a structure close to that of a locally Kundt structure. 
	\begin{theorem}[\cite{zeghib_symmetric}]\label{Theorem: Dynamics}
		Let $M$ be a homogeneous Lorentzian manifold with a non-compact isotropy group. Then, $M$ contains a lightlike totally geodesic  hypersurface (and, by homogeneity, such a hypersurface passes through any point of $M$).
	\end{theorem}
	
	We follow the proof in \cite{zeghib_symmetric}. Let $H$ denote the isotropy group of a point $ p  \in M$.  Consider an isometry $f \in H$ and its graph, $\mathsf{Graph}(f) \subset M \times M$. According to the next lemma, $\mathsf{Graph}(f)$ is an isotropic, totally geodesic $d$-dimensional submanifold of $M \times M$, equipped with the metric $g \oplus (-g)$ (where $d := \dim M)$. Recall that an isotropic submanifold is one for which the metric restricted to the tangent bundle is identically zero.
	
	\begin{lemma}
		Let $(M,g)$ be a semi-Riemannian manifold of dimension $n$. Let $f \in \mathsf{Isom}(M,g)$, and $S := \mathsf{Graph}(f) = \{(x, f(x)), x \in M\}$. Then $S$ is an isotropic totally geodesic submanifold of $(M \times M, g \oplus (-g))$.
	\end{lemma}
	\begin{proof}
		We have $T_{(p,f(p))} S = \{(v,d_pf(v)), v \in T_p X\}$. If $V = (v,d_pf(v))$ is tangent to $S$, then the geodesic in $(M \times M, g \oplus (-g))$ tangent to $V$ is given by $\gamma = (\gamma_1,\gamma_2)$, with $\gamma_1$ the geodesic in $M$ such that $\gamma_1'(0)=v$, and $\gamma_2$ the geodesic in $M$ such that $\gamma_2'(0)=d_p f(v)$. Since $f$ is an isometry, we have $\gamma_2 =f \circ \gamma_1$. Therefore, $\gamma$ lies in $S$, showing that $S$ is totally geodesic. The fact that $S$ is isotropic for $g \oplus (-g)$ is straightforward.	
	\end{proof}
	In fact, the property that $S$ is totally geodesic also follows from the more general  observation below:
	\begin{fact}
		Let $(M, h)$ be a semi-Riemannian manifold of dimension $n = 2d$ and index $d$. An isotropic submanifold $S$ of $M$ of dimension $d$ is totally geodesic.
	\end{fact}
	\begin{proof}
		Let $X \in \Gamma(T S )$. We claim that for any $Y \in \Gamma(T S), g(\nabla_XX,Y)=0$.
		To prove this, write: $h(\nabla_X X,Y)=\nabla_X h(X,Y)-h(X,\nabla_X Y)$. Since $S$ is isotropic, $h(X, Y ) = 0$ for all $X, Y \in \Gamma(TS)$, so the first term of the sum vanishes. Thus, $h(\nabla_X X,Y)=-h(X,\nabla_X Y)$. Next, since $\nabla_X Y - \nabla_Y X \in \Gamma(T S)$, we have $h(\nabla_XY - \nabla_Y X,X) = 0$. This implies   $h(\nabla_X Y,X) = h(\nabla_Y X,X)$, which yields $h(\nabla_X X,Y)=-h(\nabla_Y X,X)$. But $\nabla_Y h(X,X) = 2h(\nabla_Y X,X) = 0$; our claim follows. This shows that $\nabla_X X \in \Gamma(T^\perp S)$. Given that the index is $2d = n$, we have $T^\perp S = TS$. Therefore, $\nabla_X X \in \Gamma(T S)$ for all $X \in \Gamma(TS)$, proving that $S$ is totally geodesic.	
	\end{proof}
	
	\noindent Now, let $f_n \in H$ be a diverging sequence, meaning it has no convergent sub-sequence. Consider the graphs $S_n:=\mathsf{Graph}(f_n)$. By compactness of the Grassmannian of $d$-dimensional subspaces of $T_p M \times T_p M$, we can find a limit $L$ of a subsequence of $S_n$. \\

	\noindent To give a formal meaning to this, consider a small convex neighborhood $C$ of $(p, p)$ in $M \times M$. This means that any two points in $M$ can be joined by a unique geodesic segment contained in $M$. Consider $S_n \cap C$ and let $S_n^0$ denote the connected component of $(p, p)$ in $S_n \cap C$. Now, one can give sense to the convergence of the graphs $S_n$ exactly as in the situation of affine subspaces in an affine flat space. More precisely, $S_n^0$ converge to $L$ if the tangent spaces $T_{(p,p)} S_n^0$ converge to $T_{(p,p)} L$.\\

	\noindent Such a limit $L$ is an isotropic, totally geodesic submanifold of $M \times M$ of dimension equal to $\dim M$. However, $L$ is no longer the graph of some map $f : M \to M$, since otherwise $f_n$ would converge to $f$. \\

	\noindent Thus, the intersection of $L$ with $\{p\} \times M$ is non-trivial, but has at most dimension $1$, because the intersection is isotropic and $M$ is Lorentzian. Therefore, the projection $L'$ of $L$ onto $M \times \{p\}$ is a totally geodesic hypersurface in $M \times \{p\}$. \\

	\noindent To show that $L'$ is lightlike, consider vectors $(X,Y)$ and $(0,Y_0) \in T L$.  Then $(X,Y-Y_0) \in TL$ is isotropic for $g \oplus (-g)$, which implies that $g(Y,Y_0)=0$ for any $Y \in p_2(TL)$.  Moreover, we obtain $g(X,X)=g(Y,Y) \geq 0$.  In particular, the projection $p_2(T L)$ spans a lightlike subspace, of dimension $\geq n-1$. Thus, $L'$, which has dimension $n-1$, must contain  a lightlike vector, as otherwise, $p_2(TL)$ would contain a subspace of dimension $ n-1$ on which the metric is positive definite, contradicting the presence of a lightlike vector orthogonal to it. Hence, $L'$ is  a lightlike subspace. \medskip

	\noindent This completes the proof of Theorem \ref{Theorem: Dynamics}.\bigskip

\textbf{Comment:} We have shown that on a homogeneous Lorentzian manifold $M$ with a non-compact isotropy group, there exists a lightlike totally geodesic  hypersurface through every point of $M$. However, in order for $M$ to admit a locally Kundt structure, there must be a collection of hypersurfaces with a local product structure. This occurs, for example, if the hypersurface through a point $p \in M$ (and hence, throught any point of $M$) is unique. More generally, an interesting question is to investigate under which conditions this setup leads to a locally Kundt structure\blfootnote{The first author is fully supported by the grants, PID2020-116126GB-I00\\
		(MCIN/ AEI/10.13039/501100011033), and the framework IMAG/ Maria de Maeztu,\\ 
		CEX2020-001105-MCIN/ AEI/ 10.13039/501100011033.}.

\bibliographystyle{abbrv}
\bibliography{Bibliography}

\begin{thebibliography}{10}

\bibitem{Blau}
M.~Blau and M.~O'Loughlin.
\newblock Homogeneous plane waves.
\newblock {\em Nuclear Physics B}, 654(1):135--176, 2003.

\bibitem{boucetta}
M.~Boucetta, A.~Meliani, and A.~Zeghib.
\newblock Kundt three-dimensional left invariant spacetimes.
\newblock {\em Journal of Mathematical Physics}, 63(11), 2022.

\bibitem{coley2}
J.~Brannlund, A.~Coley, and S.~Hervik.
\newblock Supersymmetry, holonomy and {K}undt spacetimes.
\newblock {\em Classical and Quantum Gravity}, 25(19):195007, 2008.

\bibitem{coley1}
A.~Coley, S.~Hervik, G.~Papadopoulos, and N.~Pelavas.
\newblock Kundt spacetimes.
\newblock {\em Classical and Quantum Gravity}, 26(10):105016, 2009.

\bibitem{coley3}
A.~Coley, R.~Milson, N.~Pelavas, V.~Pravda, A.~Pravdov{\'a}, and
  R.~Zalaletdinov.
\newblock Generalizations of pp-wave spacetimes in higher dimensions.
\newblock {\em Physical Review D}, 67(10):104020, 2003.

\bibitem{ergodic}
M.~Einsiedler and T.~Ward.
\newblock Ergodic theory: with a view towards {N}umber {T}heory.
\newblock 2011.

\bibitem{foulon}
P.~Foulon and B.~Hasselblatt.
\newblock Contact {A}nosov flows on hyperbolic 3--manifolds.
\newblock {\em Geometry \& Topology}, 17(2):1225--1252, 2013.

\bibitem{ghys_flotsAnosov}
{\'E}.~Ghys.
\newblock Flots d'{A}nosov dont les feuilletages stables sont
  diff{\'e}rentiables.
\newblock In {\em Annales scientifiques de l'Ecole normale sup{\'e}rieure},
  volume~20, pages 251--270, 1987.

\bibitem{Walker}
P.~Gilkey, M.~Brozos-Vazquez, E.~Garcia-Rio, S.~Nik{\v{c}}evi{\'c}, and
  R.~V{\'a}squez-Lorenzo.
\newblock {\em The geometry of {W}alker manifolds}.
\newblock Springer Nature, 2022.

\bibitem{godbillon}
C.~Godbillon.
\newblock {\em Dynamical systems on surfaces}.
\newblock Springer Science \& Business Media, 2012.

\bibitem{article1}
M.~Hanounah, L.~Mehidi, and A.~Zeghib.
\newblock On homogeneous plane waves.
\newblock {\em arXiv preprint arXiv:2311.07459}, 2023.

\bibitem{kulkarni}
R.~S. Kulkarni and F.~Raymond.
\newblock 3-dimensional {L}orentz space-forms and {S}eifert fiber spaces.
\newblock {\em Journal of Differential Geometry}, 21(2):231--268, 1985.

\bibitem{leistner2016completeness}
T.~Leistner and D.~Schliebner.
\newblock Completeness of compact {L}orentzian manifolds with abelian holonomy.
\newblock {\em Mathematische Annalen}, 364(3):1469--1503, 2016.

\bibitem{mehidi2022completeness}
L.~Mehidi and A.~Zeghib.
\newblock On completeness and dynamics of compact {B}rinkmann spacetimes.
\newblock {\em arXiv preprint arXiv:2205.07243}, 2022.

\bibitem{BO}
B.~O'Neill.
\newblock Semi-{R}iemannian geometry, with applications to {R}elativity.
\newblock {\em Pure and Applied Math}, 103, 1983.

\bibitem{penrose}
R.~Penrose and W.~Rindler.
\newblock {\em Spinors and space-time: Volume 2, Spinor and twistor methods in
  space-time geometry}, volume~2.
\newblock Cambridge University Press, 1984.

\bibitem{exact-solutions}
H.~Stephani, D.~Kramer, M.~MacCallum, C.~Hoenselaers, and E.~Herlt.
\newblock {\em Exact solutions of {E}instein's field equations}.
\newblock Cambridge university press, 2009.

\bibitem{thurston_existence}
W.~P. Thurston.
\newblock Existence of codimension-one foliations.
\newblock {\em Annals of Mathematics}, 104(2):249--268, 1976.

\bibitem{zeghibgeo}
A.~Zeghib.
\newblock Geodesic foliations in {L}orentz {$3$}-manifolds.
\newblock {\em Comment. Math. Helv.}, 74(1):1--21, 1999.

\bibitem{zeghib-lipschitz}
A.~Zeghib.
\newblock Lipschitz regularity in some geometric problems.
\newblock {\em Geometriae Dedicata}, 107:57--83, 2004.

\bibitem{zeghib_symmetric}
A.~Zeghib.
\newblock Remarks on {L}orentz symmetric spaces.
\newblock {\em Compositio Mathematica}, 140(6):1675--1678, 2004.

\end{thebibliography}
	
\end{document}